\documentclass[twoside,12pt]{article}
\usepackage{amsmath}
\usepackage{amsthm}
\usepackage{amssymb}
\usepackage{amscd}
\usepackage[british]{babel}
\usepackage{graphicx}
\usepackage{times}
\usepackage{comment}
\usepackage{url}
\usepackage{color}

\newcommand{\RR}{{\mathbb R}}
\newcommand{\CC}{{\mathbb C}}
\newcommand{\DD}{{\mathbb D}}

\newcommand{\TT}{{\mathbb T}}
\newcommand{\ZZ}{{\mathbb Z}}
\newcommand{\NN}{{\mathbb N}}

\def\ii{\mathrm{i}}
\def\ee{\mathrm{e}}

\def\im{\mbox{\rm Im}}

\newcommand{\norm}[1]{|{#1}|}
\newcommand{\snorm}[1]{\lVert#1\rVert}     

\def\dif{{ \rm d}}

\def\R{{\mathcal R}}
\def\cO{{\mathcal O}}

\def\NF{N}
\def\eps{\varepsilon}
\newcommand{\INF}{\mathcal{I}_{\NF}}
\def\hrho{{\hat\rho}}
\def\trho{{\tilde\rho}}
\newcommand{\nuNi}[1]{\nu_{#1}} 
\newcommand{\muNi}[1]{\mu_{#1}}

\theoremstyle{plain}
\newtheorem{definition}{Definition}
\newtheorem{theorem}{Theorem}

\newtheorem{corollary}{Corollary}

\newtheorem{remark}{Remark}

\title{On the sharpness of the R\"ussmann estimates}

\author{
Jordi-Llu\'{\i}s Figueras\thanks{figueras@math.uu.se}\\
{\footnotesize Department of Mathematics} \\
{\footnotesize Uppsala University} \\
{\footnotesize Box 480, 751 06 Uppsala (Sweden).}
\and
Alex Haro\thanks{alex@maia.ub.es}\\
{\footnotesize Departament de Matem\`atiques i Inform\`atica}\\
{\footnotesize Universitat de Barcelona} \\
{\footnotesize Gran Via 585, 08007 Barcelona (Spain).}
\and
Alejandro Luque\thanks{luque@icmat.es}\\
{\footnotesize Instituto de Ciencias Matem\'aticas} \\
{\footnotesize Consejo Superior de Investigaciones Cient\'{\i}ficas} \\
{\footnotesize C/ Nicol\'as Cabrera 13-15, 28049 Madrid (Spain).}}

\begin{document}
\maketitle

\thispagestyle{empty}

\begin{abstract}
Estimating the norm of the solution of the linear difference equation
$u(\theta)-u(\theta+\omega)=v(\theta)$ plays
a fundamental role in KAM theory.  Optimal (in certain
sense) estimates for the solution of this equation were provided by R\"ussmann
in the mid 70's. The aim of this paper is to compare the sharpness of these
classic estimates with more specific estimates obtained with the help of the
computer. We perform several experiments to quantify the improvement
obtained when using computer assisted estimates.
By comparing these estimates with the actual norm of the solution,
we can analyze the different sources of overestimation, thus
encouraging future improvements.
\end{abstract}


\noindent \emph{Keywords}: KAM theory; cohomologic equations;
R\"ussmann estimates; small divisors.


\newpage

\section{Introduction}\label{sec:intro}

Given a zero-average periodic function $v: \TT^n \rightarrow \CC$ and a vector
$\omega\in\RR^n$, we consider the linear difference equation
\begin{equation}\label{eq:russmann}
u(\theta)-u(\theta+\omega)=v(\theta)\,.
\end{equation}
This cohomological equation plays a central role in KAM theory, and arises for
example in the proof of the existence of invariant tori for symplectic maps
(see \cite{Moser62,Russmann76b} and \cite{Arnold63a,Arnold63b,
BroerHS96,Llave01,Kolmogorov54,Moser67} for other contexts). A typical KAM
scheme consists in performing a Newton-like iteration process such that
Equation~\eqref{eq:russmann} has to be solved at every iteration
step. The convergence of the procedure in a certain scale of Banach spaces of regular
functions (endowed with a suitable norm) is ensured by estimating the norm of
the solution of \eqref{eq:russmann} in such scale.  The optimality of these
estimates is crucial for applying the KAM theorem in specific examples and
determining the size of the perturbation for which invariant tori
persist\footnote{The interested reader is referred to Section 1.4
in~\cite{CellettiC07} for a brief history and references of the application of
KAM theory, and to~\cite{CellettiC97,Locatelli98,LocatelliG00} for computer
assisted proofs in problems of celestial mechanics.}.  These optimal estimates
were provided by R\"ussmann in the celebrated paper \cite{Russmann76a}
(see~\cite{Russmann75} for the study of the analogous linear differential
equation).

A computer-assisted methodology to apply the KAM theory in particular problems,
based on the so-called a posteriori approach, 
has been recently presented in~\cite{FiguerasHL}. 
A direct application of this methodology permitted to
prove the existence of the golden invariant curve for the standard map
up to a perturbation parameter $\eps=0.9716$, which is less than 0.004\% below
the breakdown threshold observed numerically
(see~\cite{Greene79} for a semi-numeric criteria and
\cite{CellettiC88,LlaveR90,LlaveR91} for previous computer assisted proofs).
One of the technical improvements presented in~\cite{FiguerasHL}
is the use of sharper R\"ussmann estimates, obtained with the help
of the computer. The idea behind these estimates is to compute
explicitly a finite number of divisors, and use Diophantine properties
to control the remaining ones.

The aim of this paper is to present a suitable and detailed illustration of the
improved R\"ussmann estimates used in~\cite{FiguerasHL}, and to compare this
estimates both with the classic R\"ussmann ones and with a good numerical
approximation of the solution. To this end, we use suitably chosen test functions and discuss
the dependence on the different parameters of the problem. Hence, we quantify
the improvement of the computer assisted estimates with respect to the classic
ones, and also measure the error with respect to the actual norm.

The content of the paper is organized as follows.
In Section~\ref{sec:setting} we formalize the problem outlined above.
After recalling some elementary notation on analytic
periodic functions and norms,
we revisit the R\"ussmann estimates to control the
regularity of the solutions of Equation~\eqref{eq:russmann}.
We introduce some convenient notation to analyze the different
sources of overestimation. Finally, we include some
useful estimates to control the norm of an analytic
function using discrete Fourier transform.
In Section~\ref{sec:norm} we describe a methodology to enclose the norm
of an analytic function.  We illustrate quantitatively that the computation of
a finite number of small divisors is enough to 
numerically capture the analytic norm of the solutions.
In Section~\ref{sec:numerical1} we present some numerical studies of the
behavior of the norm of the solutions of Equation~\eqref{eq:russmann}, and we
compare it with the R\"ussmann estimates. Due to the large number of
parameters in the problem, we consider several situations separately.
In Section~\ref{sec:conclusions} we summarize the conclusions of our
study, and present some conjectures derived from the
experiments in Section~\ref{sec:numerical1},
thus encouraging future research.

\section{Notation and basic results}\label{sec:setting}

To present the previous statements in a precise way, we need to introduce some
notation regarding analytic functions on the torus.  We 
use standard notation for the real torus $\TT^n=\RR^n/\ZZ^n$ and the closed complex disk
$\bar \DD = \{ z \in \CC\,:\, |z| \leq 1\}$.
A complex strip of  $\TT^n$ of width $\rho>0$ is defined
as
\[
\TT^{n}_{\rho}= 
\left\{\theta\in \CC^n / \ZZ^n \ : \ \norm{\im{\,\theta_{i}}} < \rho, 
\,i=1,\dots, n\right\}.
\]

We consider analytic functions $u:\TT^n_\rho\to\CC$ continuous
up to the boundary of $\TT^n_{\rho}$, and endow these functions with the norm
\begin{equation}\label{eq:norm1}
\snorm{u}_{\rho}  = 
\sup_{\theta \in \TT^n_{\rho}}  \norm{u(\theta)}.
\end{equation}
Moreover, we denote the Fourier expansion by
\begin{equation}\label{eq:fourier:coef}
u(\theta)=\sum_{k \in \ZZ^n} \hat u_k \ee^{2\pi \ii k \cdot \theta}, \qquad \hat u_k =
\int_{[0,1]^n} u(\theta) \ee^{-2 \pi \ii k \cdot \theta} \dif \theta.
\end{equation}

\subsection{Estimates for the solution of the cohomological equation}
\label{ssec:estimates}

If $v$ is a function with zero average ($\hat v_0 = 0$) and 
$\omega$ is ergodic, then Equation~\eqref{eq:russmann}
has a formal solution $u:=\R v$ given by
\begin{equation}\label{eq:small:formal}
\R v(\theta) = \sum_{k \in \ZZ^n \backslash \{0\} } \hat u_k \mathrm{e}^{2\pi
\mathrm{i} k \cdot  \theta}, \qquad \hat u_k = \frac{\hat v_k}{1-\mathrm{e}^{2\pi
\mathrm{i}  k \cdot \omega}}.
\end{equation}
Notice that all solutions of Equation~\eqref{eq:russmann}
differ by a constant, since $\hat u_0$ is free.

Due to the effect of the
small divisors $1-\mathrm{e}^{2\pi \mathrm{i}  k \cdot \omega}$ in
Equation~\eqref{eq:small:formal}, it is well known that ergodicity is not enough to
ensure regularity of the
solution.
Regularity is obtained by imposing non-resonant 
conditions on $\omega$.
In this
paper, we consider the following classic condition:
\begin{definition}\label{def-Diophantine}
Given $\gamma>0$ and $\tau\geq n$, we say that $\omega\in\RR^n$ is a
$(\gamma, \tau)$-Diophantine vector of frequencies
if
\begin{equation}\label{eq:Diophantine}
\norm{k \cdot \omega-m}\ge \,\gamma\, |k|_1^{-\, \tau}, 
\qquad  \forall k\in\ZZ^n\backslash\{0\}, \, m\in\ZZ , 
\end{equation}
where $|k|_1 = \sum_{i= 1}^n |k_i|$.
\end{definition}
Notice that if $\omega$ is $(\gamma, \tau)$-Diophantine, then it is also
$(\tilde \gamma, \tilde \tau)$-Diophantine with $\tilde \gamma \leq \gamma$ and
$\tilde \tau \geq  \tau$. The set of $(\gamma, \tau)$-Diophantine vectors has
positive Lebesgue measure if $\tau>n$.

The following result provides sufficient conditions to
control, with explicit estimates, the norm of the solution of
Equation~\eqref{eq:russmann}:

\begin{theorem}\label{lem-Russmann}
Let $\omega\in \RR^n$ be a $(\gamma,\tau)$-Diophantine frequency vector, for
certain $\gamma>0$ and $\tau\geq n$.
Then, for any analytic function $v: \TT_\rho^n \rightarrow \CC$, with
$\snorm{v}_\rho < \infty$ and $\rho>0$,  there exists a unique zero-average
analytic solution 
$u:  \TT_\rho^n \rightarrow \CC$ of~\eqref{eq:russmann}, with $u=\R v$.
Moreover, given $L\in \NN$, for any $0<\delta<\rho$ we have 
\[
\snorm{u}_{\rho-\delta} \leq 
\frac{c_R(\delta)}{\gamma\delta^{\tau}}
\snorm{v}_{\rho},
\]
where
\begin{equation}\label{eq:russmann:mejor}
c_R(\delta) =\sqrt{\gamma^2 \delta^{2\tau} 2^{n}
\sum_{0<|k|_1\leq L} \frac{\mathrm{e}^{-4\pi |k|_1\delta}}{4 |\sin(\pi k
\cdot \omega)|^2}  + 2^{n-3} \zeta(2,2^\tau) (2\pi)^{-2\tau}
\int_{4\pi\delta (L+1)}^\infty u^{2\tau} e^{-u}\ \dif u },
\end{equation}
and $\zeta(a,b)=\sum_{j\geq 0} (b+j)^{-a}$ is the Hurwitz zeta function.
\end{theorem}

\begin{proof}[Sketch of the proof]
The proof of this result, fully discussed in~\cite{FiguerasHL}, 
follows using the same arguments originated
in~\cite{Russmann75,Russmann76a}, but with an eye in the feasibility of
computing rigorous upper bounds of finite sums. For the sake of completeness,
we discuss the sketch of the proof in order to point out several inequalities
which will be analyzed later. We omit the details that are not related with
the numerical explorations presented in this paper.


We note that the divisors in the expression~\eqref{eq:small:formal} are written
as
\[
| 1-\mathrm{e}^{2\pi \mathrm{i} k \cdot \omega}  | =  2 |\sin(\pi
k \cdot \omega)|.
\]
Then, the norm of $\R v$ is controlled as
\begin{equation}\label{eq:normF}
\|\R v\|_{\rho-\delta} \leq \sum_{k \in
\ZZ^n\backslash \{0\}} \frac{|\hat v_k| \mathrm{e}^{2 \pi |k|_1
(\rho-\delta)}}{2 |\sin(\pi k \cdot \omega)|},
\end{equation}
and using Cauchy-Schwarz inequality we obtain
\begin{equation}\label{eq:Cauchy-Schwarz}
\|\R v\|_{\rho-\delta} \leq \bigg( \sum_{k
\in \ZZ^n\backslash \{0\}} |\hat v_k|^2 \mathrm{e}^{4\pi |k|_1\rho}
\bigg)^{1/2} \bigg( \sum_{k \in \ZZ^n\backslash \{0\}} \frac{\mathrm{e}^{-4 \pi
|k|_1\delta}}{4|\sin(\pi k \cdot \omega)|^2 } \bigg)^{1/2}.
\end{equation}

On the one hand, the first term 
in \eqref{eq:Cauchy-Schwarz}
is bounded by
\begin{equation}\label{eq:Bessel}
\sum_{k \in \ZZ^n\backslash \{0\}} |\hat v_k|^2 \mathrm{e}^{4 \pi |k|_1\rho} \leq 
\sum_{k \in \ZZ^n} |\hat v_k|^2 \mathrm{e}^{4 \pi |k|_1\rho} \leq
2^n \|v\|^2_\rho
\end{equation}
(see details in~\cite{Russmann75}), and on the other hand,
we split the second term as
\begin{equation}\label{eq:summa}
\sum_{k \in \ZZ^n\backslash \{0\}} \frac{\mathrm{e}^{-4 \pi |k|_1\delta} }{4
|\sin(\pi k \cdot \omega)|^2} = \sum_{0<|k|_1\leq L} \frac{\mathrm{e}^{-4 \pi
|k|_1\delta}}{4 |\sin(\pi k \cdot \omega)|^2} + \sum_{|k|_1 > L}
\frac{\mathrm{e}^{-4 \pi |k|_1\delta}}{4 |\sin(\pi k \cdot \omega)|^2}.
\end{equation}
The finite sum of the elements $0<|k|_1\leq L$ can be evaluated
for the selected frequency vector $\omega$. Then, the tail is controlled
using the standard R\"ussmann argument.
\end{proof}

We highlight the following cases:
\begin{itemize}
\item \emph{Classic R\"ussmann estimates}: 
In this case, $c_R(\delta)$ is uniformly estimated by a
constant $c_R^0$ independent of $\delta$. Indeed, 
for $L=0$ we have
\begin{align}
c_R(\delta) = {} &
\bigg(
2^{n-3} \zeta(2,2^\tau) (2\pi)^{-2\tau}
\int_{4\pi\delta}^\infty u^{2\tau} e^{-u}\ \dif u
\bigg)^{1/2} \\
\leq {} &
\bigg(
2^{n-3} \zeta(2,2^\tau) (2\pi)^{-2\tau} \Gamma(2\tau +1) 
\bigg)^{1/2}
=: c_R^0\,.
\label{eq:classic:russmann}
\end{align}
This is the
classic R\"ussmann constant (see~\cite{Russmann75,Russmann76a}).

\item \emph{Ad hoc R\"ussmann estimates}: 
Using a computer it is standard to obtain a sharp rigorous upper
bound of the expression for $c_R(\delta)$ given in
Equation~\eqref{eq:russmann:mejor}.  To this end, we enclose $\omega$ with an
interval vector $\varpi$ (such that $\sin(\pi k \cdot \omega) \neq 0$ for every
$\omega \in \varpi$ and for every $|k|_1 \leq L$), and we rigorously enclose
the finite sum for $0<|k|_1 \leq L$ using interval arithmetics. Moreover, we
consider upper bounds of the integral in the tail using that, if $y>x$,
\[
\int_{y}^\infty u^x e^{-u} \dif u \leq \frac{y}{y-x} y^x e^{-y}.
\]
Applying this last estimate requires to take $L$ such that $2\pi\delta(L+1)>
\tau$.  This approach has been used in~\cite{FiguerasHL}, where we refer to the
reader for additional implementation details.
\end{itemize}

\begin{remark}\label{rem:polys}
The estimates presented above are optimal in the sense of the asymptotic
dependence of the divisor $\gamma \delta^{\tau}$ (former works in the
literature used the much pessimistic factor $\gamma \delta^{n+\tau}$). This
follows by studying functions with a single harmonic $v_k(\theta):=\ee^{2 \pi
\ii k \cdot \theta}$. A direct computation yields the estimate
\begin{equation}\label{eq:R1harm}
\snorm{\R v_k}_{\rho-\delta} \leq \frac{\ee^{2 \pi |k|_1 (\rho-\delta)}}{2 |\sin(\pi k \cdot \omega)|} 
\leq \frac{1}{4 \gamma} |k|_1^\tau \ee^{-2\pi |k|_1 \delta} \snorm{v_k}_\rho 
\leq \frac{1}{4 \gamma} \left(\frac{\tau}{2\pi \ee \delta} \right)^\tau \snorm{v_k}_\rho \,.
\end{equation}
This computation also shows that R\"ussmann estimates may provide
a large overestimation for a fixed function in the above family. As R\"ussmann himself
observed in \cite{Russmann75}, the estimate in~\eqref{eq:R1harm} seems to be the best possible estimate
in the class of classic estimates, and he was able to obtain such an estimate for the case $n=1$
combining his approach with the theory of continued fractions.
As far as we know, this question still remains open for $n>1$.
\end{remark}

Motivated by the above discussion, we introduce a functional
$F_{\rho,\delta,\omega}$ acting on $v$ as follows
\begin{equation}\label{eq:operF}
F_{\rho,\delta,\omega} v := \frac{\snorm{\R v}_{\rho-\delta}}{\snorm{v}_\rho} \leq
\snorm{\R}_{\rho,\rho-\delta}\,,
\end{equation}
where $\snorm{\R}_{\rho,\rho-\delta}$ is the norm of the operator $\R$ acting between the
corresponding spaces.
Then, it turns out that the R\"ussmann estimates provide upper bounds for the
image of the previous functional:
\begin{equation}\label{eq:inequalities}
F_{\rho,\delta,\omega} v 
\leq \frac{c_R(\delta)}{\gamma \delta^\tau}
\leq \frac{c_R^0}{\gamma \delta^\tau}
\,.
\end{equation}

In Section~\ref{sec:numerical1} we are going to study numerically the sharpness of these two inequalities as a
function of $\rho$, $\delta$, $\omega$, and $v$. 

\subsection{Analyzing the different sources of
overestimation}\label{ssec:sources}

Assume that we take $L$ large enough in such a way that the
contribution of the tail in Equation~\eqref{eq:russmann:mejor} can be
neglected. In this case, it is interesting to analyze the different sources of
the overestimation given by Theorem~\ref{lem-Russmann}.  

We denote
\begin{equation}\label{eq:def:ineq1}
1 \leq I_R := \frac{c_R(\delta)}{\gamma \delta^\tau} \frac{1}{F_{\rho,\delta,\omega}v} 
= \bigg(2^{n-2} \sum_{k \in
\ZZ^n\backslash \{0\}} \frac{\mathrm{e}^{-4 \pi
|k|_1\delta}}{|\sin (\pi k \cdot \omega)|^2 } \bigg)^{1/2} \frac{\snorm{v}_\rho}{
\snorm{\R v}_{\rho-\delta}}
\,,
\end{equation}
which stands for the overestimation produced by the \emph{ad hoc} R\"ussmann estimates.

It is clear that the expression of $I_R$ breaks down in terms of three
different sources of overestimation as $I_R =I_1 I_2 I_3$, where
\begin{align}
I_1 := {} & 
\frac{\sum_{k \in
\ZZ^n\backslash \{0\}}\frac{|\hat v_k| \mathrm{e}^{2 \pi |k|_1
(\rho-\delta)}}{2 |\sin(\pi k \cdot \omega)|}}{
\snorm{\R v}_{\rho-\delta}} \,, \label{eq:c1} \\
I_2 := {} &
\frac{
\sqrt{
\left(\sum_{k \in \ZZ^n\backslash \{0\}}
|\hat v_k|^2 \mathrm{e}^{4 \pi |k|_1\rho}\right) \left( \sum_{k \in
\ZZ^n\backslash \{0\}} \frac{\mathrm{e}^{-4 \pi
|k|_1\delta}}{4|\sin (\pi k \cdot \omega)|^2 } \right)}}{
\sum_{k \in
\ZZ^n\backslash \{0\}}\frac{|\hat v_k| \mathrm{e}^{2 \pi |k|_1
(\rho-\delta)}}{2 |\sin(\pi k \cdot \omega)| } }
\,, \label{eq:c2} \\
I_3 := {} & \frac{2^{n/2}  \snorm{v}_\rho}{
\sqrt{\sum_{k \in \ZZ^n\backslash \{0\}}
|\hat v_k|^2 \mathrm{e}^{4 \pi |k|_1\rho}}}\,. \label{eq:c3}
\end{align}

Notice that $I_1, I_2, I_3$ correspond to 
inequalities \eqref{eq:normF}, \eqref{eq:Cauchy-Schwarz} and 
\eqref{eq:Bessel} respectively. Since 
\[
\frac{\log(I_1)}{\log(I_R)}
+
\frac{\log(I_2)}{\log(I_R)}
+
\frac{\log(I_3)}{\log(I_R)}
=1\,,
\]
we can make use of a color chart to represent the contribution (fraction in
logarithmic scale) to each element to the total overestimation
(see Figure~\ref{fig:F1:golden_v0:error} for an example).

Moreover, we observe that the expression in~\eqref{eq:def:ineq1} also breaks
down in two factors. On the one hand, the factor
\[
\frac{c_R(\delta)}{\gamma \delta^\tau} 
 = \bigg(2^{n-2} \sum_{k \in
\ZZ^n\backslash \{0\}} \frac{\mathrm{e}^{-4 \pi
|k|_1\delta}}{|\sin (\pi k \cdot \omega)|^2 } \bigg)^{1/2}
\]
is independent of $v$ and is decreasing with respect to $\delta$, and on the
other hand, the factor
\begin{equation}\label{eq:factor2}
\frac{1}{F_{\rho,\delta,\omega}v}
 = \frac{\snorm{v}_\rho}{\snorm{\R v}_{\rho-\delta}}
\end{equation}
is increasing with respect to $\delta$. 

\subsection{Using discrete Fourier transform to approximate analytic functions}\label{ssec:DFT}

Here, we briefly recall some
explicit estimates presented in~\cite{FiguerasHL} that allows us to control the
norm of an analytic function in terms of the norm of its discrete Fourier
transform.  Using this estimates we are able to rigorously validate the
numerical computations using a finite number of Fourier
coefficients.

We consider a sample of points on the regular grid of size
$\NF=(N_{1},\ldots,N_{n}) \in \NN^n$
\begin{equation}\label{eq:sample:torus}
\theta_j:=(\theta_{j_1},\ldots,\theta_{j_n})=
\left(\frac{j_1}{N_{1}},\ldots,
\frac{j_n}{N_{n}}\right),
\end{equation}
where $j= (j_1,\ldots,j_n)$, with $0\leq j_\ell < N_{\ell}$ and $1\leq \ell
\leq n$. This defines an $n$-dimensional sampling $\{u_j\}$, with
$u_j=u(\theta_j)$. 
The integrals in Equation~\eqref{eq:fourier:coef} are approximated
using the 
discrete Fourier transform:
\[
\tilde u_k= \frac{1}{N_{1} \cdots
N_{n}} \sum_{0\leq j < \NF} u_j \mathrm{e}^{-2\pi
\mathrm{i} k \cdot \theta_j},
\]
where the sum runs over integer subindices $j \in \ZZ^n$ such that $0\leq
j_\ell < N_{\ell}$ for $\ell= 1,\dots, n$.  Notice that $\tilde u_k$ is
periodic with respect to the components $k_1,\dots, k_n$ of $k$, with
periods $N_{1}, \dots, N_{n}$, respectively.
The periodic function $u$ is approximated by the discrete Fourier approximation
\begin{equation}\label{eq:four:approx}
\tilde u(\theta)= \sum_{k \in \INF} \tilde u_k \ee^{2 \pi \mathrm{i} k \cdot
\theta},
\end{equation}
where
\begin{equation}\label{eq:finite}
 \INF= \bigg\{ k \in \ZZ^n \,|\, -\frac{N_{\ell}}{2} \leq k_\ell <
 \frac{N_{\ell}}{2}, 1\leq \ell \leq n \bigg\}.
\end{equation}
is a finite set of multi-indices. 

\begin{theorem}\label{theo-DFT}
\label{ADFT} Let $v:\TT^{n}_{\trho} \to \CC$ be an analytic
function for $\trho>0$ and continuous up to the boundary.
Let $\tilde
v$ be the discrete Fourier approximation of $u$ in the regular grid of size
$\NF= (N_{1},\dots, N_{n}) \in \NN^n$. Then
\begin{equation}\label{eq:DFT}
\snorm{\tilde v-v}_\rho \leq C_{\NF}(\rho, \trho) \snorm{v}_{\trho},
\end{equation}
for $0\leq \rho < \trho$,
where 
\[
C_{\NF}(\rho, \trho)= S_\NF^{*1}(\rho,\trho) + S_\NF^{*2}(\rho,\trho)  +
T_\NF(\rho,\trho)
\]
is given by
\[
S_\NF^{*1}(\rho,\trho) = 
\prod_{\ell= 1}^n \frac{1}{1-\ee^{-2\pi  \trho N_{\ell} }}
\sum_{\begin{array}{c} \sigma\in \{-1,1\}^n \\ \sigma\neq (1,\dots,1)
\end{array}}
\prod_{\ell= 1}^n \ee^{(\sigma_\ell-1)\pi\trho N_{\ell}}
\nuNi{\ell}(\sigma_\ell\trho-\rho),
\]
\[
S_\NF^{*2}(\rho,\trho) = \prod_{\ell= 1}^n \frac{1}{1-\ee^{-2\pi
\trho N_{\ell} }}  \left(1- \prod_{\ell= 1}^n  \left(1-\ee^{-2\pi
\trho N_{\ell} }\right)\right) \prod_{\ell= 1}^n \nuNi{\ell}(\trho-\rho)
\]
and
\[
T_\NF(\rho,\trho)= \left(  \frac{\ee^{2\pi (\trho-\rho)} + 1}{\ee^{2\pi
(\trho-\rho)} -1} \right)^n \ \left( 1 - \prod_{\ell= 1}^n \left(1-
\muNi{\ell}(\trho-\rho)\ e^{-\pi(\trho-\rho) N_{\ell}} \right) \right),
\]
with
\[
\nuNi{\ell}(\delta)= \frac{\ee^{2\pi \delta} + 1 }{\ee^{2\pi \delta} -1}
\left(1- \muNi{\ell}(\delta) \ \ee^{-\pi \delta N_{\ell}}\right) \qquad
\mbox{and} \qquad \muNi{\ell}(\delta) = 
\begin{cases}  
\ 1 &\mbox{if $N_{\ell}$ is even} \\ \displaystyle \frac{2
\ee^{\pi\delta}}{\ee^{2\pi\delta}+1} &\mbox{if $N_{\ell}$ is odd}
\end{cases}.
\]
Notice that $C_{\NF}(\rho, \trho)$ satisfies
$C_{\NF}(\rho, \trho)=\cO(\ee^{-\pi (\trho-\rho) \min_\ell N_\ell})$.
\end{theorem}

We refer the reader to~\cite{FiguerasHL} for the proof and implementation
details. By combining Theorem~\ref{lem-Russmann} and Theorem~\ref{theo-DFT} we
obtain a direct way to control $F_{\rho,\delta, \omega}v$ for a given function
$v$. 

\begin{corollary}\label{coro}
Let $\omega\in \RR^n$ be a $(\gamma,\tau)$-Diophantine frequency vector, for
certain $\gamma>0$ and $\tau\geq n$.
Let $v: \TT_\hrho^n \rightarrow \CC$ be an analytic function for $\hrho>0$.
Then, for any $0<\delta \leq \rho< \tilde \rho< \hrho$ we have
the following interval enclosure
\[
F_{\rho,\delta,\omega} v =
\frac{\snorm{\R v}_{\rho-\delta}}{\snorm{v}_\rho} \in 
\left[
\frac
{\snorm{\R \tilde v}_{\rho-\delta} - \frac{c_R(\delta) C_{\NF}(\rho,
\tilde \rho)}{\gamma\delta^{\tau}} \snorm{v}_{\tilde \rho}}
{\snorm{\tilde v}_{\rho} + C_{\NF}(\rho,
\tilde \rho) \snorm{v}_{\tilde \rho}}
,
\frac{\snorm{\R \tilde v}_{\rho-\delta} + \frac{c_R(\delta) C_{\NF}(\rho,
\tilde \rho)}{\gamma\delta^{\tau}} \snorm{v}_{\tilde \rho}}
{\snorm{\tilde v}_{\rho} - C_{\NF}(\rho,\tilde \rho)
\snorm{v}_{\tilde \rho}}
,
\right] \,,
\]
where $c_R(\delta)$ is given by Equation~\eqref{eq:russmann:mejor}.
\end{corollary}

\section{On the numerical computation of the analytic norm}\label{sec:norm}

Let us describe the method used in this paper to numerically compute the
analytic norm of a periodic function of the form
\begin{equation}\label{eq:FAM1}
v_{s,\{a_k\}} (\theta)= \sum_{k \in \ZZ^n \backslash \{0\}} \hat v_k \ee^{2\pi \ii k \cdot \theta}\,,
\qquad
\hat v_k = a_k \frac{\ee^{-2\pi |k|_1 \hat \rho}}{|k|_1^s}\,,
\end{equation}
where $\hat \rho>0$, and $a_k \in \bar \DD$.  We will denote $v_s$
if $a_k=1$ for every $k\in \ZZ^n \backslash \{0\}$. We will also denote
$v_{s,+}$ if $a_k=1$ for every $k\in \ZZ^n \backslash \{0\}$ such that $k_i\geq
0$, and $a_k=0$ otherwise. These functions are analytic in the complex strip
$\TT^n_{\hat \rho}$. Moreover, the Fourier coeficients of $v_{s,\{a_k\}}$ decay
as $|a_k| |k|^{-s}$ at the boundary of $\TT^n_{\hat \rho}$.

The idea consists in approximating the Fourier series of the considered
function using the support $\INF$, given in~\eqref{eq:four:approx}. Given a
function $v:=v_{s,\{a_k\}}$, defined for certain $\hrho>0$, we choose a number
$\rho < \hrho$, and we select a discretization $\INF$ in order to
guarantee a precision $\eps$ in the numerical approximation of the norm
$\snorm{v}_\rho$. Notice that this discretization will keep the required
tolerance when computing $\snorm{v}_{\rho-\delta}$, for every $0<\delta \leq
\rho$.

\begin{remark}
Notice that one must be very careful with round-off errors. The norm
$\snorm{v}_\rho$ is obtained by computing the maximum of the function $v$
restricted to the boundary $\partial \TT^n_\rho$. To this end, we must multiply
the Fourier coefficients $\hat v_k$ by exponentials of the form $\ee^{\pm 2 \pi
(k_1 \pm \cdots \pm k_n)\rho}$. Since the coefficients $\hat v_k$ decay as
$\ee^{-2\pi |k|_1 \hrho}$, then the mentioned multiplications produce very
large round-off errors, specially if $\rho$ is close to $\hrho$. For this
reason, the factors $\ee^{-2\pi |k|_1 \hrho}$ are written separately in the
definitions of~\eqref{eq:FAM1}.
\end{remark}

In particular, we take a uniform grid $N=(M,\ldots,M)$ where $M=2^q$ is the
smallest power such that
\[
\frac{\ee^{-\pi (\hrho-\rho) M /2}}{(M/4)^s} < \eps\,.
\]
Then, we evaluate the function $v=v_{s,\{a_k\}}$ on the boundary $\partial \TT^n_\rho$ 
thus obtaining $2^n$ samples of $\# \INF$ elements, 
perform the backward fast Fourier transform associated to these
$2^n$ samples, and select the maximum value from the obtained
numbers. This value is a good approximation, with error $\cO(1/M)$, of the maximum of the
function at the boundary. Finally, we refine the computation of this maximum using the Newton
method, thus obtaining $\snorm{\tilde v}_{\rho}$ the norm associated
with the discrete Fourier transform of $v$ in $\INF$. Then, the
true norm is enclosed in the interval (see Theorem~\ref{theo-DFT})
\[
\snorm{v}_\rho 
\in
\Big[
\snorm{\tilde v}_\rho 
-
C_\NF(\rho,\trho) \snorm{v}_\trho
~~,~~
\snorm{\tilde v}_\rho 
+
C_\NF(\rho,\trho) \snorm{v}_\trho
\Big]
\]
for $\rho < \trho < \hrho$. In the following computations, we select the number
$\trho$ that minimizes the length of such interval. The norm
$\snorm{v}_\trho$ is simply overestimated analytically.

To illustrate the above methodology, we consider the function $v_s$, given by~\eqref{eq:FAM1} with $a_k=1$.
In addition to the fact that we have an accurate control of the decay of the Fourier coefficients, this
family has the property that the norm $\snorm{v_s}_\rho$ can be evaluated
explicitly. Our aim is to convince the reader that, even though we
consider test functions with infinitely many harmonics, we can choose a
suitable discretization providing a good description of the norm with a
finite amount of computations. This may seem obvious at a first glance, but we
have to take into account the effect of the infinitely many small divisors
(see Theorem~\ref{theo-DFT} and Corollary~\ref{coro}).
All computations have been performed with $30$ digits 
(using the \texttt{MPFR} library~\cite{RevolR05}).

We first consider the case $n=1$. 
In the following computations we use the simple observation that a periodic
function having real Fourier coefficients all with the same sign attains its
maximum at the point $\theta=0$.
Then, we compute $\snorm{v_0}_\rho$ for $0\leq \rho < \hrho$, as follows:
\[
\snorm{v_0}_\rho = \sum_{k>0} \ee^{-2\pi k (\hrho-\rho)}
+
\sum_{k>0} \ee^{-2\pi k (\hrho+\rho)}
=
\frac{\cosh(2\pi \rho)-\ee^{-2\pi \hrho}}{\cosh(2\pi
\hrho)-\cosh(2\pi \rho)}.
\]
Notice that $\snorm{v_0}_\rho \rightarrow \infty$ when $\rho \rightarrow
\hrho$. In order to consider the case $\rho=\hrho$ we assume that $s>1$
and observe that
\[
\snorm{v_s}_\rho 
=
\sum_{k>0} 
\frac{\ee^{-2\pi k (\hrho-\rho)}}{k^s}
+
\sum_{k>0} 
\frac{\ee^{-2\pi k (\hrho+\rho)}}{k^s} 
=
\mathrm{Li}_s(\ee^{-2\pi(\hrho-\rho)})
+
\mathrm{Li}_s(\ee^{-2\pi(\hrho+\rho)})\,,
\]
where $\mathrm{Li}_s(z)$ is the polylogarithm function.
In Table~\ref{tab:norm:1d} we present some computations of the norm
$\snorm{v_0}_\rho$ using the method described. We compare with the exact value
of this norm and with the constant $C_\NF(\rho,\trho)$ that can be used to
rigorously enclose the numerical computations. We ask for a tolerance
$\eps=10^{-30}$. The computations saturate the precision of the machine. 

\begin{table}[!h]
\centering
{\scriptsize
\begin{tabular}{|c|c c c c|}
\hline
$\rho$ & $N$ & $\snorm{\tilde v_0}_\rho$ & $\frac{\snorm{\tilde
v_0}_\rho-\snorm{v_0}_\rho}{\snorm{v_0}_\rho}$ & $C_\NF(\rho,\trho)$ \\
\hline
0.0 & 64 & 3.7418731973e-03 &     0.0e+00 & 9.6e-88 \\
0.1 & 64 & 4.5099874933e-03 &    -2.0e-30 & 5.1e-79 \\
0.2 & 64 & 7.1365311745e-03 &    -4.3e-31 & 2.8e-70 \\
0.3 & 64 & 1.2735869952e-02 &     0.0e+00 & 1.5e-61 \\
0.4 & 128 & 2.3749435599e-02 &    0.0e+00 & 3.3e-105\\
0.5 & 128 & 4.5246411394e-02 &    0.0e+00 & 1.0e-87 \\
0.6 & 128 & 8.8185405159e-02 &   -1.1e-30 & 3.0e-70 \\
0.7 & 256 & 1.7903995865e-01 &    0.0e+00 & 3.8e-105\\
0.8 & 256 & 3.9785030192e-01 &   -4.9e-31 & 3.8e-70 \\
0.9 & 512 & 1.1435745379e+00 &   -3.4e-30 & 5.9e-70 \\
0.93 & 1024 & 1.8101817456e+00 &  1.3e-30 & 8.9e-98 \\
0.96 & 2048 & 3.4997999963e+00 &  1.3e-30 & 1.5e-111\\
0.99 & 8192 & 1.5420733666e+01 & -1.1e-29 & 5.5e-111\\
0.995 & 16384 & 3.1333610168e+01 &  4.8e-30 & 1.0e-110\\
0.999 & 65536 & 1.5865547020e+02 & -4.6e-29 & 1.2e-87 \\
0.9999 & 524288 & 1.5910494868e+03 & 1.5e-27 & 9.3e-69 \\
\hline
\end{tabular}
\caption{{\footnotesize 
Numerical computation of the analytic norm of the function $v_0$
with $n=1$ and $\hrho=1$, for several values of $\rho$. 
We show the relative error using
the explicit formula and we include also the minimum of
$C_\NF(\rho,\trho)$ with respect to $\trho$, with $\rho<\trho<\hrho$.
}} \label{tab:norm:1d}}
\end{table}

For the case $n=2$, we restrict to functions $v_{s,+}$, thus simplifying the
combinatorics of the computations. We have
\begin{align*}
\snorm{v_{s,+}}_\rho = {} & \sum_{\ell =1}^\infty \sum_{k_1+k_2 = \ell} 
\frac{\ee^{-2\pi \ell (\hrho-\rho)}}{\ell^s}\\
= {} & \sum_{\ell =1}^\infty (\ell+1)  
\frac{\ee^{-2\pi \ell (\hrho-\rho)}}{\ell^s}
= 
\mathrm{Li}_{s-1}(\ee^{-2\pi (\hrho-\rho)})
+
\mathrm{Li}_{s}(\ee^{-2\pi (\hrho-\rho)}).
\end{align*}
Table~\ref{tab:norm:2d} is analogous to Table \ref{tab:norm:1d} for the
function $v_{0,+}$ with $n=2$. Of course, the memory cost in the second case is
more demanding, so we do not approach the boundary as much as in the
first case.

\begin{table}[!h]
\centering
{\scriptsize
\begin{tabular}{|c|c c c c|}
\hline
$\rho$  & $N$ & $\snorm{\tilde v_{0,+}}_\rho$ & $\frac{\snorm{\tilde
v_{0,+}}_\rho-\snorm{v_{0,+}}_\rho}{\snorm{v_{0,+}}_\rho}$ & $C_\NF(\rho,\trho)$ 
\\
\hline
0.0  & 64$\times$64 & 3.7453736011e-03 &  0.0e+00 & 1.9e-87 \\
0.1  &  64$\times$64 & 7.0378103397e-03 &  4.4e-31 & 1.0e-78 \\
0.2  &  64$\times$64 & 1.3253135843e-02 &  0.0e+00 & 5.6e-70 \\
0.3  &  64$\times$64 & 2.5059548137e-02 & -4.9e-31 & 3.1e-61 \\
0.4  &   128$\times$128 & 4.7753162472e-02 & 5.2e-31 & 7.0e-105 \\ 
0.5  &  128$\times$128 & 9.2371351668e-02 &-5.3e-31 & 2.1e-87 \\ 
0.6  &  128$\times$128 & 1.8405377620e-01 &-5.4e-31 & 7.1e-70 \\ 
0.7  &   256$\times$256 & 3.9008106334e-01 &-1.0e-31 & 1.0e-104 \\ 
0.8  &  256$\times$256 & 9.5395121039e-01 & 1.2e-30 & 1.3e-69 \\ 
0.9  &  512$\times$512 & 3.5948837747e+00 & 0.0e+00 & 3.9e-69 \\ 
0.93 & 1024$\times$1024 & 6.8970910155e+00 & -2.2e-30 & 8.2e-97 \\
0.96 &  2048$\times$2048 & 1.9248159655e+01 & 2.6e-30 & 2.4e-110 \\
\hline
\end{tabular}
\caption{{\footnotesize 
Numerical computation of the analytic norm of the function $v_{0,+}$
with $n=2$ and $\hrho=1$, for several values of $\rho$. 
Implementation details are the same as in 
Table~\ref{tab:norm:1d}.}}
\label{tab:norm:2d}}
\end{table}

In Table~\ref{tab:norm:s} we present some computations of the norm
$\snorm{v_s}_\rho$ at the boundary, i.e. taking $\rho=\hrho$. We use the same
implementation parameters as before. We consider the cases $n=1$ and $n=2$
in the same table.
The numbers illustrate
the dependence of the computational cost on the regularity of the function.
Notice that controlling the tail using $C_N(\rho,\rho)$ does not make sense.
To this end, it is not difficult to extend the arguments in Section~\ref{ssec:DFT}
to consider $\mathcal{C}^r$-functions. In this case, one uses that the
decay of Fourier coefficients of a $\mathcal{C}^r$-function is of the form
$|\hat f_k| \leq (2\pi k)^{-r} \snorm{f}_{\mathcal{C}^r}$. 

\begin{table}[!h]
\centering
{\scriptsize
\begin{tabular}{|c|c c c|c c c|}
\hline
& \multicolumn{3}{c|}{$n=1$} & \multicolumn{3}{c|}{$n=2$} \\
\hline
$s$ & $N$ & $\snorm{\tilde v_s}_\rho$ & $\frac{\snorm{\tilde
v_s}_\rho-\snorm{v_s}_\rho}{\snorm{v_s}_\rho}$
& $N$ & $\snorm{\tilde v_{s,+}}_\rho$ & $\frac{\snorm{\tilde
v_{s,+}}_\rho-\snorm{v_{s,+}}_\rho}{\snorm{v_{s,+}}_\rho}$ \\
\hline
15 & 512 & 1.0000340756e+00 &  0.0e+00 & 512$\times$512 & 2.0000918364e+00 &
-1.5e-30 \\
14 & 1024 & 1.0000647355e+00 & 0.0e+00 & 1024$\times$1024 & 2.0001839615e+00 &
0.0e+00 \\
13 & 1024 & 1.0001262007e+00 & -7.8e-31 & 1024$\times$1024 & 2.0003687999e+00
& -2.5e-29 \\
12 & 2048 & 1.0002495739e+00 & 7.8e-31 & 2048$\times$2048 & 2.0007402752e+00 &
-7.0e-30 \\
11 & 4096 & 1.0004976759e+00 & 7.8e-31 & & & \\
10 & 4096 & 1.0009980625e+00 & -8.8e-29 & & & \\
9 & 16384 & 1.0020118802e+00 & -2.3e-30 & & & \\
8 & 32768 & 1.0040808435e+00 & -5.7e-29 & & & \\
7 & 131072 & 1.0083527647e+00 & -1.3e-28 & & & \\
6 & 524288 & 1.0173465493e+00 & -5.0e-27 & & & \\
\hline
\end{tabular}
\caption{{\footnotesize 
Numerical computation of the analytic norm at $\rho=\hrho$ of the function $v_s$
with $\hrho=1$, for several values of $s$. 
We show the relative error using
the explicit formula.
}} \label{tab:norm:s}}
\end{table}

We observe from Tables~\ref{tab:norm:1d}, \ref{tab:norm:2d} and
\ref{tab:norm:s} that the accuracy of the computation 
when the required number of Fourier coefficients
becomes too large. This fact is probably related to the fact that the FFT
algorithm does not minimize the round-off errors.

\section{Numerical experiments}\label{sec:numerical1}

Along this section we present several numerical explorations
to illustrate the sharpness of the inequalities
in~\eqref{eq:inequalities}. Since $F_{\rho,\delta,\omega} v$
depends on multiple parameters, we organize the computations in different
subsections.
All computations in this section have been performed with $30$ digits.

\subsection{The case of the golden number}\label{ssec:1d:golden}

Here we consider the special case where $\omega = \tfrac{\sqrt{5}-1}{2}$ is the
golden mean. 
It is well known (see
e.g.~\cite{CellettiC07}) that this number 
has Diophantine constants
\[
\gamma = \frac{3-\sqrt{5}}{2}, \qquad \tau=1\,.
\]

In Figure~\ref{fig:golden_cR} we show the values of $c_R(\delta)$ and $c_R^0$
associated to this number.  In the left plot we use the value $\tau=1$ and
observe how $c_R(\delta)$ improves the bound given by the
classic constant $c_R^0$.  In the right plot we compare $c_R(\delta)$ and
$c_R^0$ using $\tau=1.2$. This choice is very interesting from the KAM point of
view since it ensures a positive measure of Diophantine numbers in a
neighborhood of $\omega$. In this case, we observe that the improvement of
\emph{ad hoc} estimates $c_R(\delta)$ increases when $\delta$ is small. 

\begin{figure}[!t]
\centering
\includegraphics[scale=0.45]{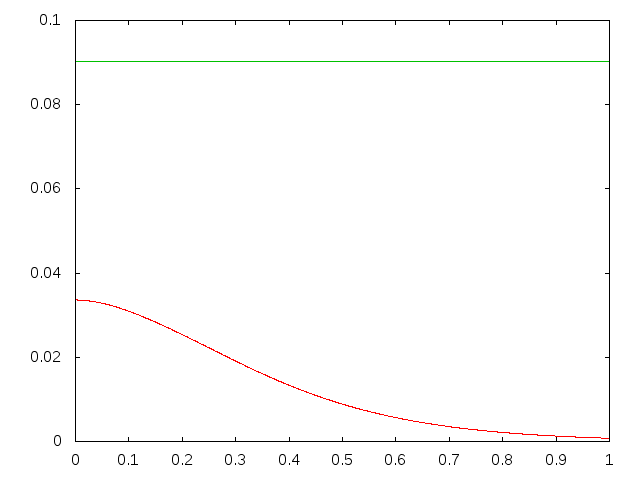}
\includegraphics[scale=0.45]{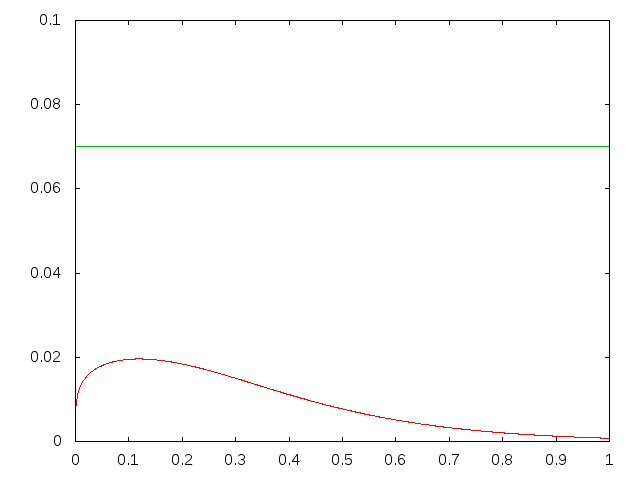}
\caption{{\footnotesize
Left ($\tau=1$) and right ($\tau=1.2$) plots: we show the curve
$\delta \mapsto c_R(\delta)$ (in red), and the constants $c_R^0$ (in green).
}}
\label{fig:golden_cR}
\end{figure}

\begin{figure}[!t]
\centering
\includegraphics[scale=0.31]{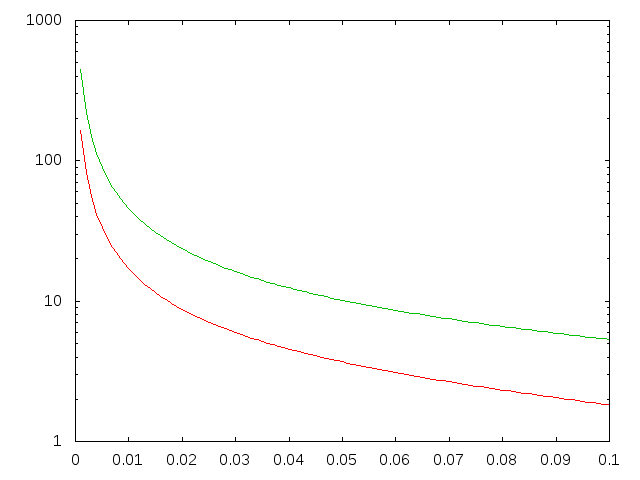}
\includegraphics[scale=0.31]{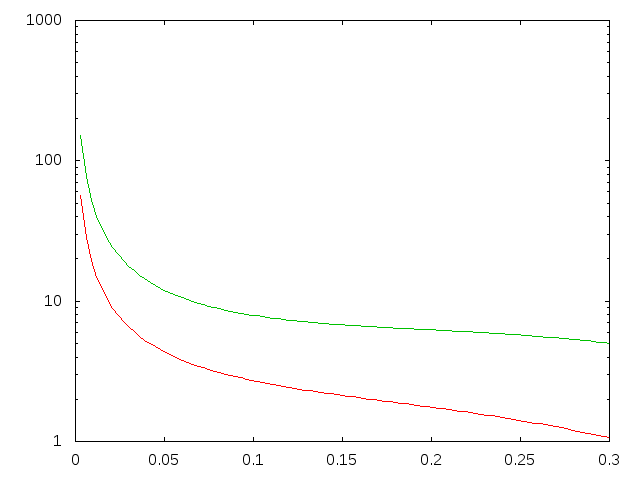}
\includegraphics[scale=0.31]{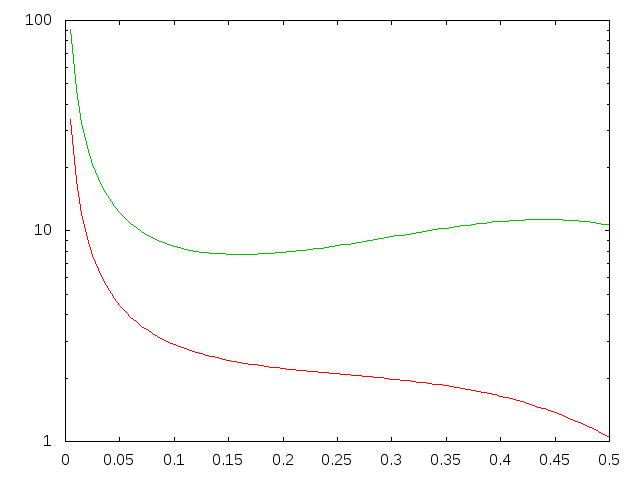}
\includegraphics[scale=0.31]{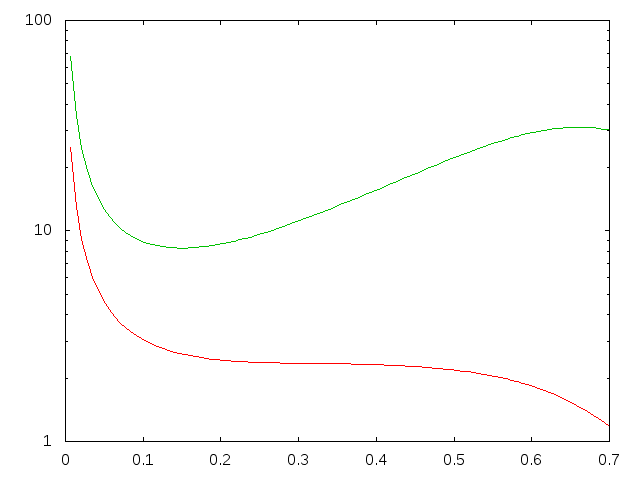}
\includegraphics[scale=0.31]{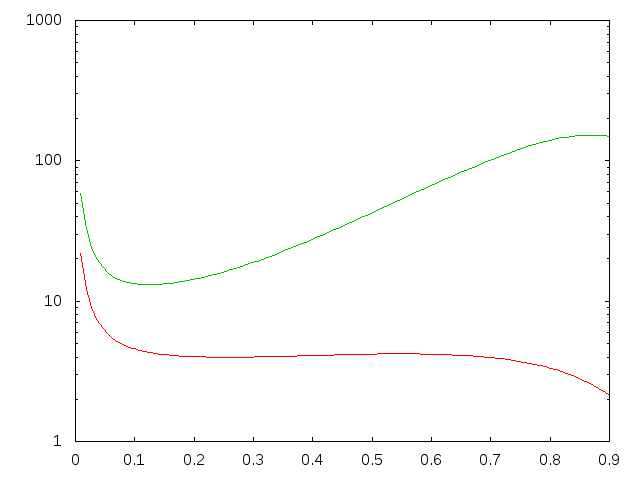}
\includegraphics[scale=0.31]{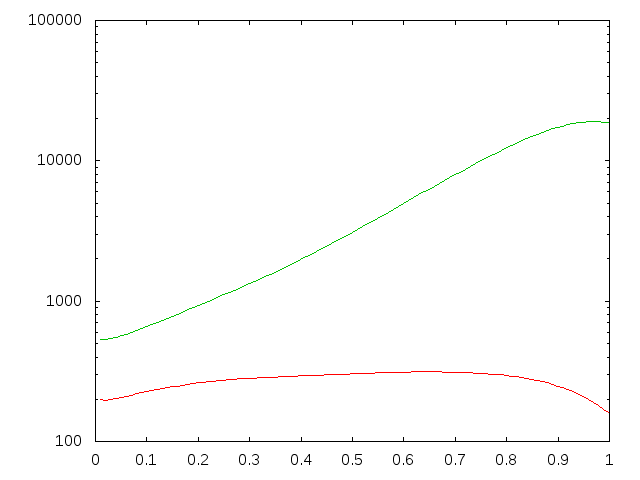}
\caption{{\footnotesize
Overestimation (in $\log_{10}$-scale) of the R\"ussmann estimates ($y$-axis) versus $\delta$ ($x$-axis).
We consider the golden rotation and the function $v_0$ with $\hat\rho=1$. 
Every plot corresponds to a different value of $\rho$ in $\{0.1, 
0.3, 
0.5, 
0.7, 
0.9, 
0.999\}$. The red curve corresponds to~\eqref{eq:ineq1}, which is the
overestimation using \emph{ad hoc} estimates. The green curve 
corresponds to~\eqref{eq:ineq2}.
}}
\label{fig:F1:golden_v0}
\end{figure}
\begin{figure}[!t]
\centering
\includegraphics[scale=0.31]{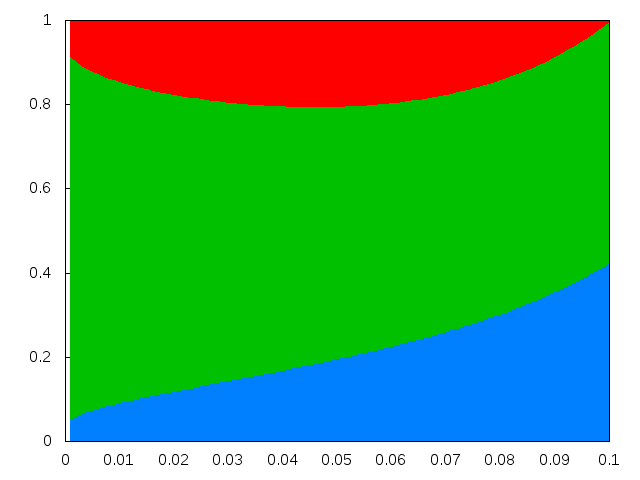}
\includegraphics[scale=0.31]{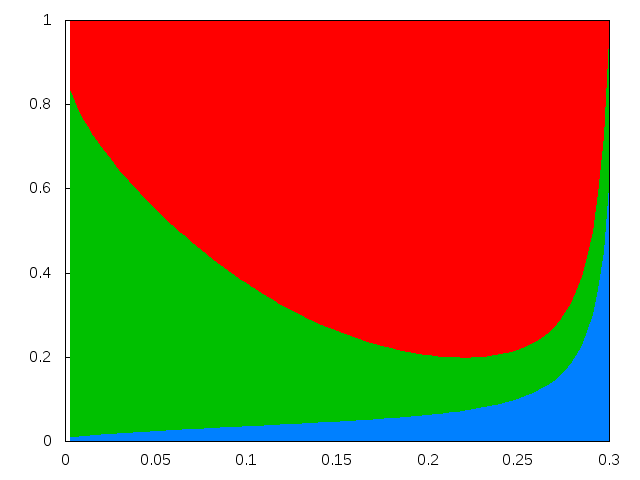}
\includegraphics[scale=0.31]{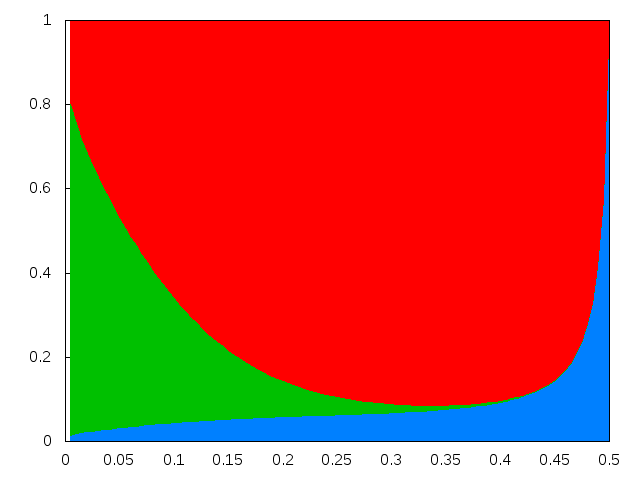}
\includegraphics[scale=0.31]{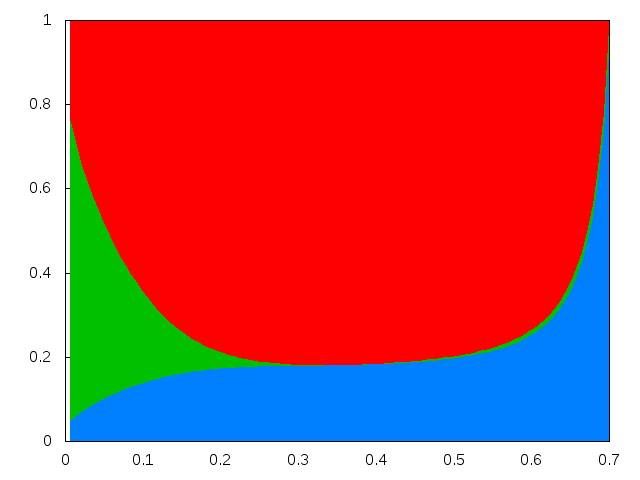}
\includegraphics[scale=0.31]{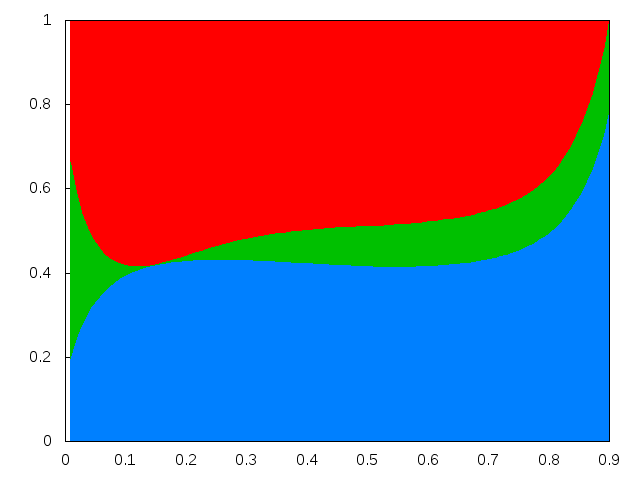}
\includegraphics[scale=0.31]{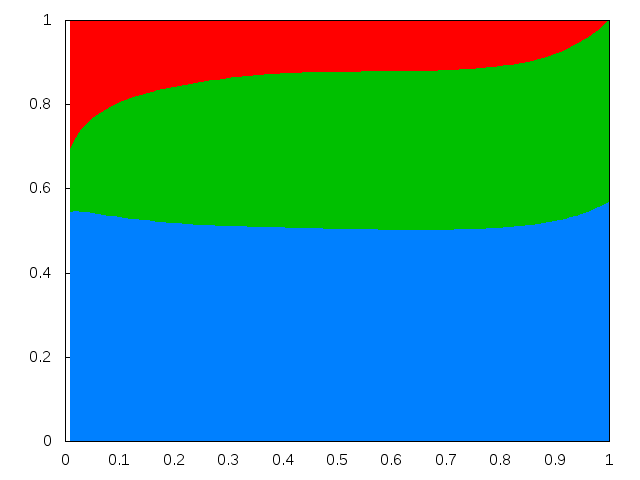}
\caption{{\footnotesize
Fraction (in logarithmic scale) of the contribution of different error sources
$I_1$, $I_2$, and $I_3$ ($y$-axis) versus $\delta$ ($x$-axis); given 
by~\eqref{eq:c1}, \eqref{eq:c2}, and~\eqref{eq:c3}, respectively.
We consider the golden rotation and the function $v_0$ with $\hat\rho=1$. 
Every plot corresponds to a different value of $\rho$ in $\{0.1, 
0.3, 
0.5, 
0.7, 
0.9, 
0.999\}$. 
The red area is 
$\tfrac{\log(I_1)}{\log(I_R)}$,
the green area is 
$\tfrac{\log(I_2)}{\log(I_R)}$,
the blue area is $\tfrac{\log(I_3)}{\log(I_R)}$.
}}
\label{fig:F1:golden_v0:error}
\end{figure}

In order to quantify how good are the R\"ussmann estimates for a given function,
we consider now the function $v_0$, given by Equation~\eqref{eq:FAM1} with $a_k=1$ for every $k \neq 0$,
and  $\hat\rho=1$.
Given a value $\rho<1$, we consider the values of $\delta = \frac{j}{100}\rho$,
for $j=1 \div 100$, and compute $F_{\rho,\delta,\omega} v_0$ and
$c_R(\delta)$. In Figure~\ref{fig:F1:golden_v0} we show the overestimation of the two inequalities
in~\eqref{eq:inequalities}, using $\log_{10}$-scale. In all computations we use
$\tau=1$.
The red curve 
corresponds to
\begin{equation}\label{eq:ineq1}
\delta \longmapsto
I_R:= \frac{c_R(\delta)}{\gamma \delta} \frac{1}{F_{\rho,\delta,\omega} v_0}\,,
\end{equation}
which stands for the overestimation produced by the \emph{ad hoc} R\"ussmann estimates.
The overestimation produced by classic R\"ussmann estimates is given by the curve
\begin{equation}\label{eq:ineq2}
\delta \longmapsto
\frac{c_R^0}{\gamma \delta} \frac{1}{F_{\rho,\delta,\omega} v_0}\,,
\end{equation}
which is plotted in green.  The results show that the use of \emph{ad hoc}
estimates $c_R(\delta)$ outperforms the classic bound $c_R^0$. In some cases
by several orders of magnitude. We observe the following:

\begin{itemize}
\item If $\rho$ is far from $\hat\rho$, the Fourier coefficients of both
$v_0(\theta \pm\ii \rho)$ and $\R v_0 (\theta \pm\ii (\rho-\delta))$
decay very fast, and the effect of the high order small
divisors is negligible. For this reason, we observe a large overestimation for
small values of $\delta$ in all cases
(see the discussion in Remark~\ref{rem:polys}).

\item The best results are obtained for intermediate values of $\rho$ (that is
$\rho \in (0.3,0.8)$ and not very small values of $\delta$ (say $\delta \in
(\rho/5,\rho]$). In these cases, we have $I_R \in (1,4]$.
We also observe that if $0 \ll \delta \approx \rho$ the ad
hoc R\"ussmann estimates provide a sharp upper bound:
$F_{\rho,\delta,\omega} v \approx c_R(\delta) \gamma^{-1} \delta^{-\tau}$.

\item If $\rho$ approaches $\hat \rho$, then the
performance of the R\"ussmann estimates deteriorates. This is because the
factor \eqref{eq:factor2} becomes very large. For example, according to
Table~\ref{tab:norm:1d}, we have $\snorm{v_0}_{0.999} \approx 159$ and
$\snorm{\R v_0}_0$ is small.
\end{itemize}

In Figure~\ref{fig:F1:golden_v0:error} we break down $I_R$ into the sources
of overestimation $I_1$, $I_2$, and $I_3$ of the \emph{ad hoc} R\"ussmann
estimates. These charts allows us to observe, in case that the overestimation
is significant, which is the relative contribution of each factor. 

\bgroup
\def\arraystretch{1.6}
\begin{table}[!h]
\centering
{\scriptsize
\begin{tabular}{|c|c c c c c c c c c c c c c c|}
\hline
$\rho$ & 0.1 & 0.2 & 0.3 & 0.4 & 0.5 & 0.6 & 0.7 & 0.8 & 0.9 & 0.93 & 0.96 &
0.99 & 0.993 & 0.999 \\
$\delta$ & 0.02 & 0.04 & 0.06 & 0.08 & 0.1 & 0.12 & 0.14 & 0.16 & 0.18 & 0.186 &
0.192 & 0.198 & 0.1986 & 0.1998 \\
Eq.~\eqref{eq:ineq1}
& 8.7 & 5.0 & 3.8 & 3.2 & 2.9 & 2.7 & 2.7 & 2.9 & 4.1 & 5.1 & 7.9 & 27.4 & 38.5
&  261.3 \\
Eq.~\eqref{eq:ineq2}
& 23.6 & 13.6 & 10.6 & 9.2 & 8.5 & 8.1 & 8.3 & 9.4 & 13.8 & 17.7 & 27.6 & 96.9 &
136.6 & 929.8 \\
$I_1$ & 1.47 & 1.78 & 1.93 & 2.00 & 2.02 & 2.02 & 2.03 & 2.09 & 2.22 & 2.27 &
2.33 & 2.39 & 2.40 & 2.41 \\
$I_2$ & 4.61 & 2.58 & 1.91 & 1.57 & 1.37 & 1.23 & 1.12 & 1.03 & 1.01 & 1.05 &
1.20 & 2.03 & 2.38 & 6.07 \\
$I_3$ & 1.29 & 1.09 & 1.04 & 1.03 & 1.05 & 1.09 & 1.17 & 1.34 & 1.81 & 2.15 &
2.83 & 5.64 & 6.74 & 17.84 \\
\hline
$\rho$ & 0.1 & 0.2 & 0.3 & 0.4 & 0.5 & 0.6 & 0.7 & 0.8 & 0.9 & 0.93 & 0.96 &
0.99 & 0.993 & 0.999 \\
$\delta$ & 0.1 & 0.2 & 0.3 & 0.4 & 0.5 & 0.6 & 0.7 & 0.8 & 0.9 & 0.93 & 0.96 &
0.99 & 0.993 & 0.999 \\
Eq.~\eqref{eq:ineq1}
& 1.8 & 1.2 & 1.1 & 1.0 & 1.0 & 1.1 & 1.2 & 1.4 & 2.1 & 2.8 & 4.5 & 16.4 & 23.2
& 159.7 \\
Eq.~\eqref{eq:ineq2}
& 5.3 & 4.2 & 5.0 & 7.0 & 10.7 & 17.4 & 30.2 & 58.7 & 150.1 & 229.9 & 430.5 & 1839.4 &
2644.8 & 18754.4 \\
$I_1$ & 1.00 & 1.00 & 1.00 & 1.00 & 1.00 & 1.00 & 1.00 & 1.00 & 1.00 & 1.00 &
1.00 & 1.00 & 1.00 & 1.00 \\
$I_2$ & 1.41 & 1.09 & 1.02 & 1.00 & 1.00 & 1.00 & 1.01 & 1.04 & 1.18 & 1.30 &
1.59 & 2.90 & 3.44 & 8.93  \\
$I_3$ & 1.29 & 1.09 & 1.04 & 1.03 & 1.05 & 1.09 & 1.17 & 1.34 & 1.81 & 2.15 &
2.83 & 5.64 & 6.74 & 17.84 \\
\hline
\end{tabular}
\caption{{\footnotesize 
Some selected computations of the R\"ussmann estimates.
We consider the golden rotation and the function $v_0$ with $\hat\rho=1$. 
The upper table corresponds to $\delta=\tfrac{\rho}{5}$ and the lower table corresponds to $\delta=\rho$.
}} \label{tab:F1:golden}}
\end{table}
\egroup

Table~\ref{tab:F1:golden} reflects quantitatively the above observations.  In
the upper table, we take different values of $\rho$ and values of $\delta$ of
the form $\delta = \tfrac{\rho}{5}$, which corresponds to a choice that is
typical when applying the KAM theorem. For $\rho \in (0.3,0.8)$ we observe that
the \emph{ad hoc} R\"ussmann estimates only overestimate the norm of $\snorm{\R
v_0}_{\rho-\delta}$ by a factor at most $4$ (typically, smaller than $3$). The
lower table corresponds to the same values of $\rho$, but we take
$\delta=\rho$. This illustrates the limit case where \emph{ad hoc} estimates
tend to be sharp (for intermediate values of $\rho$) and the difference with
respect to the classic estimates is larger.

\begin{remark}\label{rem:other:exp}
Using other functions we have obtained analogous plots as in Figure
\ref{fig:F1:golden_v0}.  In particular, we have considered a collection of
$10^5$ functions $v_{s,\{a_k\}}$, given by Equation~\eqref{eq:FAM1}, with
$\hat\rho=1$ and (uniform) random numbers $a_k \in \bar \DD$.  If one
reproduces Figure~\ref{fig:F1:golden_v0:error} for the average behavior of
these functions, then the obtained plots are almost identical to the ones
obtained for $v_0$.  Moreover, if we increase the value of the parameter $s$ we
observe that the good performance of \emph{ad hoc} R\"ussmann estimates for $0
\ll \delta \approx \rho$ prevails when $\rho \approx \hat \rho$ (notice that
$\snorm{v_s}_{\hat\rho}$ is finite for $s>1$). These figures are omitted in order
to avoid an unnecessary lengthy paper.
\end{remark}

\begin{remark}
Observe that we are dealing with functions with an infinite number
of Fourier coefficients but with an explicit control of their decay.
It is worth mentioning what happens when one perturbs these functions
by including eventually a large Fourier coefficient. For example, by
considering functions of the form
\[
v_{s,\{a_k\}} (\theta)= 
\sum_{k \in \ZZ^n \backslash \{0\}} \hat v_k \ee^{2\pi \ii k \cdot \theta}
+
\sum_{k \in \INF \backslash\{0\}} \hat w_k \ee^{2\pi \ii k \cdot \theta}
\,,
\qquad
\hat v_k = a_k \frac{\ee^{-2\pi |k|_1 \hat \rho}}{|k|_1^s}\,, 
\qquad \hat w_k \in \bar \DD\,.
\]
In this situation, one observes a poor performance of the R\"ussmann estimates:
$F_{\rho,\delta,\omega} v \ll c_R(\delta) \gamma^{-1} \delta^{-\tau}$. A
heuristic justification of this observation follows from the computation in~\eqref{eq:R1harm}.
\end{remark}

Finally, we suitably choose a family of functions that are selected to saturate
the Cauchy-Schwarz inequality in Equation~\eqref{eq:Cauchy-Schwarz}. We consider
functions of the form
\begin{equation}\label{eq:FAM3}
v_{\{b_k\}}(\theta)
= \sum_{k \in \ZZ^n\backslash \{0\} } 
\hat v_k \ee^{2\pi \ii k \cdot \theta}\,,
\qquad
\hat v_k = \frac{b_k}{\sin(\pi \omega \cdot k)} \ee^{-2\pi |k|_1 \hat \rho} \,,
\qquad
b_k \in \bar \DD\,.
\end{equation}
Notice that this family satisfies $I_2=1$ when we select
$0<\delta\leq \rho<\hrho$ such that $\hrho=\rho+\delta$.
In the left plot of Figure~\ref{fig:F3:golden_vs_disk_prop} we reproduce
the computations presented in Figures~\ref{fig:F1:golden_v0}
and~\ref{fig:F1:golden_v0:error} for the function~\eqref{eq:FAM3}
taking $b_k=1$ for every $k$. We observe a behavior that is very close to
the Family~\eqref{eq:FAM1}.
In the center plot
of Figure~\ref{fig:F3:golden_vs_disk_prop} we show the overestimation produced by the
\emph{ad hoc} R\"ussmann estimates
for $10^5$ functions of Family~\eqref{eq:FAM1} with (uniform) random numbers $a_k
\in \bar \DD$ and
$10^5$ functions of Family~\eqref{eq:FAM3} with (uniform) random numbers $b_k \in
\bar \DD$. In the range $\delta \in [\tfrac{\rho}{4.5},\tfrac{\rho}{1.5}]$ we observe
that $\min_v {I_R} \approx 2.3$. Finally, in the right plot Figure
\ref{fig:F3:golden_vs_disk_prop} we show
average behavior (including all the selected functions) of the contribution
of each source of overestimation.

\begin{figure}[!t]
\centering
\includegraphics[scale=0.31]{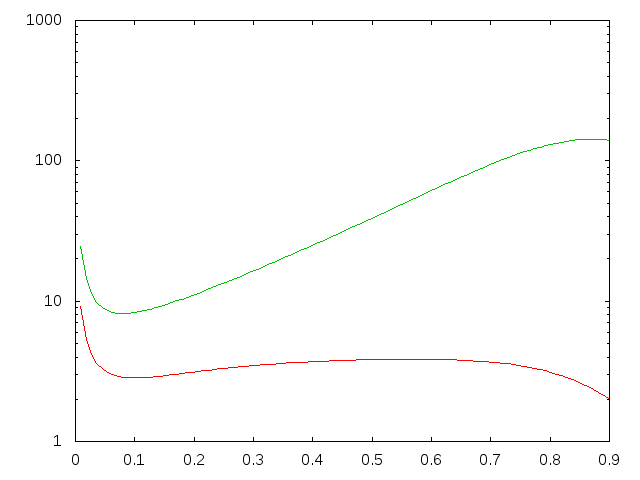}
\includegraphics[scale=0.31]{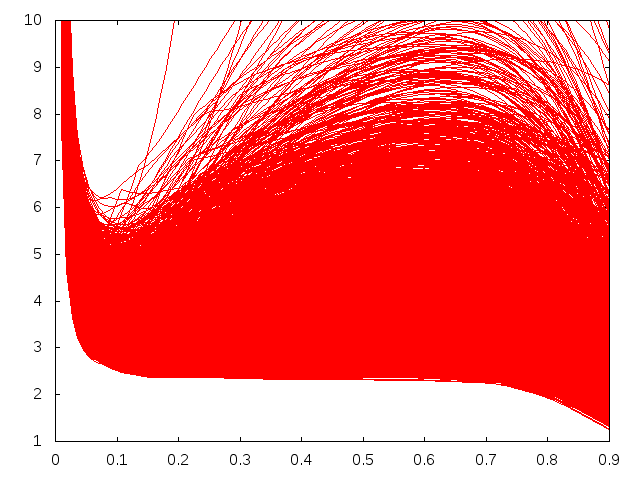}
\includegraphics[scale=0.31]{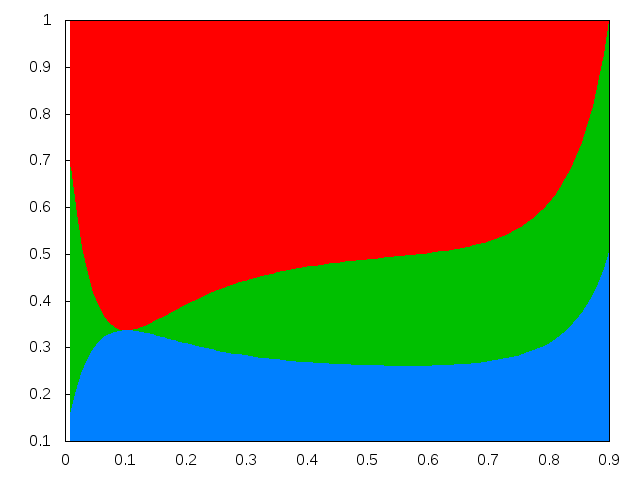}
\caption{{\footnotesize
We take $\rho=0.9$ in all plots.
Left plot: 
Overestimation (in $\log_{10}$-scale) of the R\"ussmann estimates ($y$-axis) versus $\delta$ ($x$-axis).
We consider the function $v_{\{b_k\}}$, given by~\eqref{eq:FAM3} with $b_k=1$ and $\hat\rho=1$. 
The red curve corresponds to~\eqref{eq:ineq1} and
the green curve
corresponds to~\eqref{eq:ineq2}.
Center plot:
Overestimation (in normal scale) of the \emph{ad hoc} R\"ussmann estimates ($y$-axis) versus $\delta$ ($x$-axis).
We consider $10^5$ function given by~\eqref{eq:FAM1} with $\hat\rho=1$, and 
$10^5$ function given by~\eqref{eq:FAM3} with $\hat\rho=1$. 
Right plot:
Average fraction (in logarithmic scale) of the contribution of different error sources
associated to the center plot
(see Figure~\ref{fig:F1:golden_v0:error} for color description).
}}
\label{fig:F3:golden_vs_disk_prop}
\end{figure}

\subsection{Dependence on $\omega$}\label{ssec:1d:other:rot}

Let us consider now the dependence of the R\"ussmann estimates on $\omega$.
We select the numbers
\[
\omega = \sin \left(\frac{0.02+0.5\,j}{10000}\right)\,,
\qquad 
\mbox{for $j = 1 \div 10^4$}\,.
\]
These numbers behave like Diophantine numbers and are eventually ``close'' to
low order resonances. In order to associate a pair $(\gamma,\tau)$ for each
value of $\omega$ we use two different approaches:
\begin{description}
\item [Method 1:]
Approximate $\omega$ by a quadratic number $\omega_Q$ as follows
\[
\omega \simeq \omega_Q = [a_0,a_1,\ldots,a_Q,1,1,1,\ldots]
= [a_0,a_1,\ldots,a_Q,1^\infty]\,,
\]
where $a_i$ are obtained computing the truncated continued fraction of $\omega$.
Then we
take $\tau=1$ and compute the constant $\gamma$ associated to $\omega_Q$
(see details in~\cite[Appendix B]{CellettiC07}). 

\item [Method 2:]
We enclose $\omega$ into a tiny interval and
we assign to this interval a pair $(\gamma,\tau)$ that ensures 
that the relative measure of $(\gamma,\tau)$-Diophantine numbers
is positive. We refer to~\cite[Section 4.1]{FiguerasHL} for details.
\end{description}

We first restrict the analysis to the
function $v_0$, given by Equation~\eqref{eq:FAM1} with $a_k=1$ for every $k
\neq 0$, and we fix the parameters $\delta=0.1$, $\rho=0.5$, and $\hat\rho=1$. 
We have obtained analogous results for other functions of the family (selecting
$a_k$ randomly).
In Figure~\ref{fig:F1:move_omega_prop} we plot
\begin{equation}\label{eq:ineq1:omega}
\omega \longmapsto
I_R:= \frac{c_R(\delta)}{\gamma \delta^\tau} \frac{1}{F_{\rho,\delta,\omega} v_0}
\end{equation}
in red (first and second plots), and
\begin{equation}\label{eq:ineq2:omega}
\omega \longmapsto
\frac{c_R^0}{\gamma \delta^\tau} \frac{1}{F_{\rho,\delta,\omega} v_0}
\end{equation}
in green (first plot).
It is worth mentioning that~\eqref{eq:ineq1:omega} seems to
be a curve which is regular in the sense of Whitney. Moreover,
this curve  is independent of the values
of $(\gamma,\tau)$ if $L$ is taken large enough.
On the contrary, the classic
estimates present a ``bad'' dependence on $\omega$, 
in the sense that they produce a large overestimation quite often
(between 100 and 10000 times larger). 
Let us compare also the sensitivity of classic R\"ussmann
estimates with respect to the method used to obtain $(\gamma,\tau)$.
In general, Method 2 provides better estimates than
Method 1, capturing also the effect of resonances. 

\begin{figure}[!t]
\centering
\includegraphics[scale=0.45]{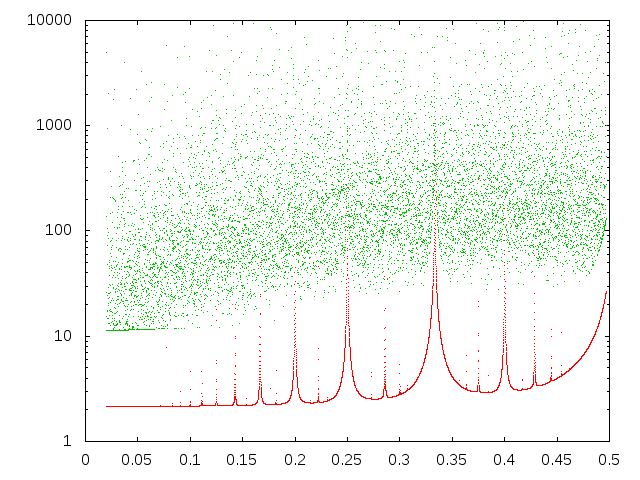}
\includegraphics[scale=0.45]{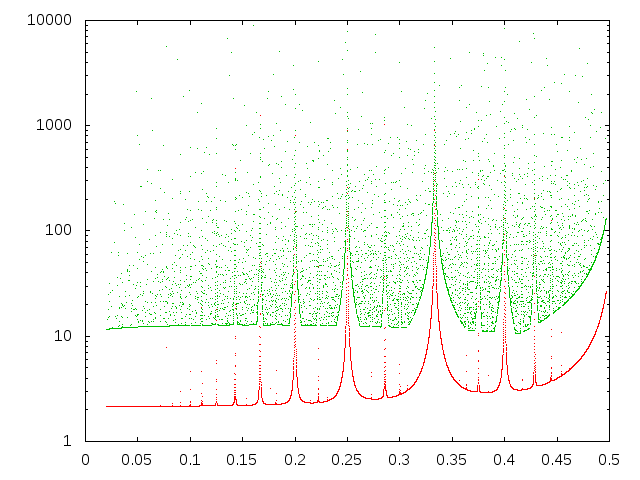}
\caption{{\footnotesize:
Overestimation (in $\log_{10}$-scale) of the R\"ussmann estimates ($y$-axis) versus $\omega$ ($x$-axis).
We consider the function $v_0$, given by~\eqref{eq:FAM1} with $\hat\rho=1$. 
We take $\rho=0.5$ and $\delta=0.1$ in all plots. 
We show~\eqref{eq:ineq1:omega} in red and~\eqref{eq:ineq2:omega} in green.
In the left plot we use Method 1 and in the right plot we use Method 2.
}}
\label{fig:F1:move_omega_prop}
\end{figure}

\begin{figure}[!t]
\centering
\includegraphics[scale=0.45]{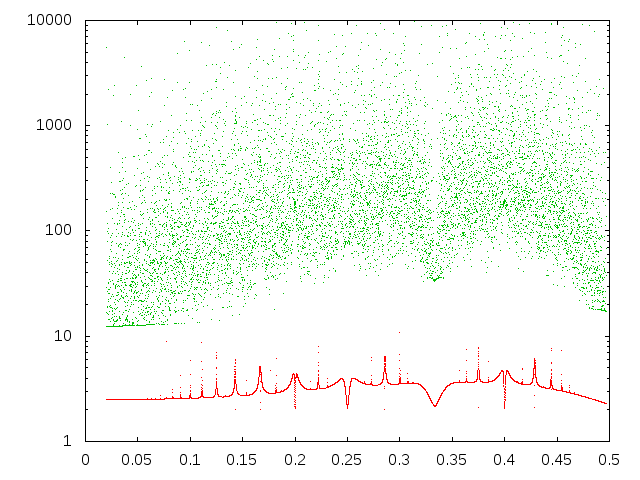}
\includegraphics[scale=0.45]{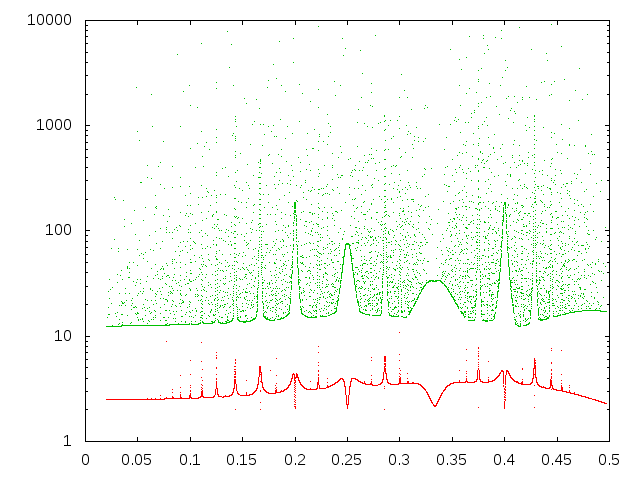}
\caption{{\footnotesize:
Overestimation (in $\log_{10}$-scale) of the R\"ussmann estimates ($y$-axis) versus $\omega$ ($x$-axis).
We consider $10^{4}$ function $v_{\{b_k\}}$, given by~\eqref{eq:FAM3} with $\hat\rho=1$
and random numbers $b_k \in \bar \DD$. 
We take $\rho=0.5$ and $\delta=0.1$ in all plots. 
We show~\eqref{eq:ineq1:omega} in red and~\eqref{eq:ineq2:omega} in green.
In the left plot we use Method 1 and in the right plot we use Method~2.
}}
\label{fig:F3:move_omega_disk_prop}
\end{figure}

Analogous computations are presented in
Figure~\ref{fig:F3:move_omega_disk_prop}. In this case, we consider $10^4$
functions of the form~\eqref{eq:FAM3} by taking (uniform) random numbers $b_k \in
\bar \DD$. For each value of $\omega$, we solve the cohomological equations and
obtain the R\"ussmann estimates for all these functions. 
We have obtained similar results for all the considered functions.
In the plots we
show the average behavior observed. The results are analogous to the ones
obtained for the function $v_0$ given by~\eqref{eq:FAM1},
the only remarkable difference being the behavior
observed for ``bad'' Diophantine numbers (i.e. close to resonances). In this case, due to the particular
relationship between the size of Fourier coefficients and the small divisors, the improvement
of the \emph{ad hoc} R\"ussmann estimates becomes even larger.

In Figure~\ref{fig:F1F3:sources:error} we show the contribution of the different
sources of error in the previous computations. This illustrates that
the relative importance of $I_1$, $I_2$ and $I_3$ depends on the family of
functions under consideration. As it was discussed in
Section~\ref{ssec:1d:golden}, the difference is not significant
when $\omega$ is the golden mean. But Figure~\ref{fig:F1F3:sources:error} shows
that we can observe different scenarios for other values of $\omega$.
For example, we find that when $\omega$ is close to a resonance the overestimation $I_2$
is dominant in the left plot and irrelevant in the right plot.

\begin{figure}[!t]
\centering
\includegraphics[scale=0.45]{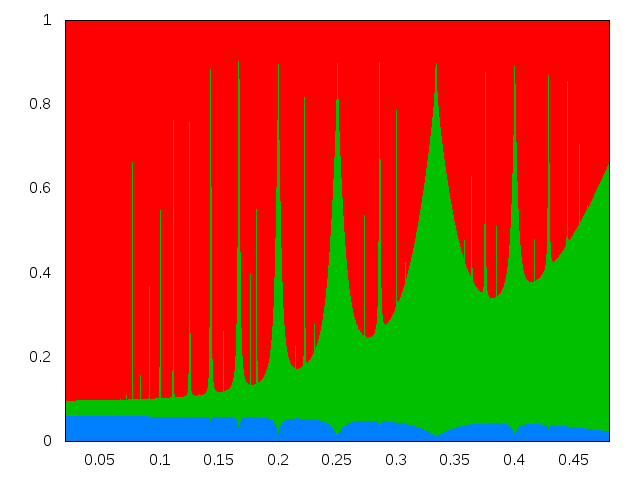}
\includegraphics[scale=0.45]{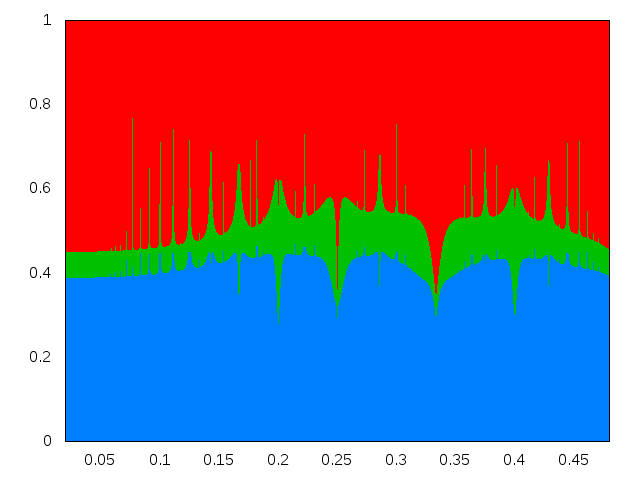}
\caption{{\footnotesize:
Average fraction (in logarithmic scale) of the contribution of different error sources
$I_1$, $I_2$, and $I_3$
($y$-axis) versus $\delta$ ($x$-axis) to the overestimation produced the \emph{ad hoc} R\"ussmann estimates.
Left plot corresponds to computations in Figure~\ref{fig:F1:move_omega_prop}, and
right plot corresponds to computations in
Figure~\ref{fig:F3:move_omega_disk_prop}.
See Figure~\ref{fig:F1:golden_v0:error} for color description.
}}
\label{fig:F1F3:sources:error}
\end{figure}


\subsection{Two dimensional case}\label{sec:numerical2}

Finally, we present some computations to quantify the performance
of R\"ussmann estimates in the case $n=2$. Now we consider
frequency vectors of the form $\omega=(\omega_1,\omega_2)$, with
\[
\omega_1= \frac{\sqrt{5}-1}{2}\,,
\qquad
\omega_2 = \sin \left(\frac{0.02+0.5\,j}{10000}\right)\,,
\qquad 
\mbox{for $j = 1\div 10^4$}\,.
\]
The behavior of the estimates
agrees qualitatively with the results previously
discussed. Hence, in order to reduce the exposition, here we
only consider the study of the function $v_0$ for 
the fixed choice of parameters $\rho=0.9$ and $\delta=0.18$.
Our aim is to illustrate quantitatively the dependence
on $\omega$ when the dimension increases.

The left plot of Figure~\ref{fig:F1:D2} is analogous to the right plot of
Figure~\ref{fig:F1:move_omega_prop}.  In this case, in order to obtain
constants $(\gamma,\tau)$ associated to $\omega$ we use 
Method 2, since Method 1 is not valid (there is no continuous
fraction expansion).  Again, we observe that \emph{ad hoc} R\"ussmann estimates depend in
a nice way on the frequency. The red points seems to be on a curve which has
some regularity in the sense of Whitney.  The overestimation of classic
estimates with respect to the \emph{ad hoc} ones is quite equivalent as in the
one dimensional case. However, we observe that the overestimation of the true
norm produced by both estimates is significantly larger, at least a factor 10.

In the right plot Figure~\ref{fig:F1:D2} we show
the three sources of overestimation. Notice that
the blue area (corresponding to $I_3$) is larger
that the blue area observed in the left plot of
Figure~\ref{fig:F1F3:sources:error}, and the main
responsible for the overestimation with respect to the case $n=1$.

\begin{figure}[!t]
\centering
\includegraphics[scale=0.45]{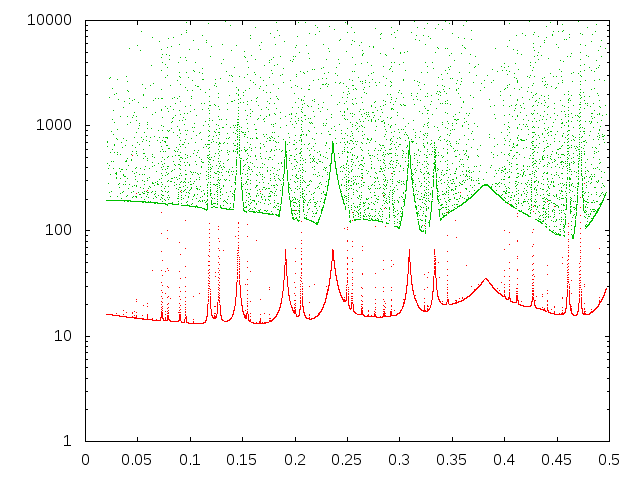}
\includegraphics[scale=0.45]{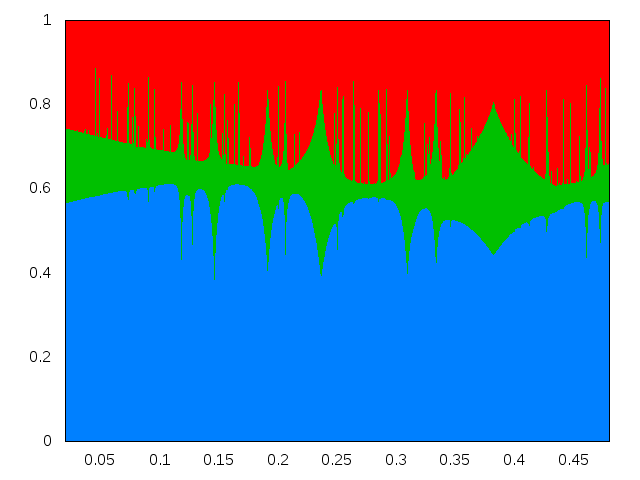}
\caption{{\footnotesize:
Left plot: Overestimation (in $\log_{10}$-scale) of the R\"ussmann estimates ($y$-axis)
versus $\omega_2$ ($x$-axis) for fixed $\omega_1$). 
We consider the function $v_0$, given by~\eqref{eq:FAM1} with $\hat\rho=1$,
$\rho=0.9$ and $\delta=0.18$.
Right plot:
contribution 
(in logarithmic scale) 
of different error sources
$I_1$, $I_2$, and $I_3$
($y$-axis) versus $\delta$ ($x$-axis) to the overestimation produced by the
\emph{ad hoc} R\"ussmann estimates in the computation of the left plot.}}
\label{fig:F1:D2}
\end{figure}

\section{Final remarks and conclusions}\label{sec:conclusions}

Here we summarize some observations and questions that arise after the presentation of the
previous results.

\begin{itemize}
\item We observe that if $0 \ll \delta \approx \rho$ and intermediates values of
$\rho$, the ad
hoc R\"ussmann estimates provide a sharp upper bound:
$F_{\rho,\delta,\omega} v \approx c_R(\delta) \gamma^{-1} \delta^{-\tau}$. This
can be readily justified by observing that the three expressions $I_1$, $I_2$,
and $I_3$ tend to saturate in this range of parameters.
Furthermore, in Figure~\ref{fig:F1:golden_v0} we observe that,
at this range of parameters, the overestimation $\delta \mapsto I_R(\delta)$
seems to be proportional to $\rho-\delta$, as $\delta \rightarrow \rho$. 
It is of interest to know if this remains true for typical functions.

\item It is clear that the three sources of overestimation $I_1$, $I_2$, and $I_3$ are
sharp independently, but in general they are not sharp jointly. For example,
$I_1$ is saturated by functions with just one harmonic or, if $n=1$,
by functions with positive Fourier coefficients only supported at $k>0$;
$I_2$ is saturated by functions of the form~\eqref{eq:FAM3} when $\rho+\delta=\hat
\rho$; and $I_3$ is saturated by functions which are almost constant at the
boundary of the complex strip $\TT_\rho$. It seems not clear how to minimize
these three overestimations simultaneously.
We have numerically observed that at certain regimes this minimization leads to
a minimum overestimation $I_R \approx 2.3$ (see right plot in
Figure~\ref{fig:F3:golden_vs_disk_prop}) for the case $n=1$ and
$\omega=\frac{\sqrt{5}-1}{2}$.

\item We have analyzed the diferent sources of overestimation $I_R=I_1 I_2 I_3$ when
a certain function or family of function is fixed. The contribution of each
source to the total overestimation depends on the parameters $\delta$, $\rho$
and $\hat\rho$. When specific regimes of the parameters or subfamilies of
functions are considered, this information could be
interesting to tailor Theorem~\ref{lem-Russmann} to improve the estimates.

\item When $\omega$ is a generic Diophantine number, the use of \emph{ad hoc} R\"ussmann estimates outperforms
the use of classic estimates by several orders of magnitude (see Section~\ref{ssec:1d:other:rot}).
Moreover, 
we observe a nice behavior of the overestimation
$I_R$ with respect to $\omega$. Actually, it seems to exhibit some regularity
(in the sense of Whitney) with respect to $\omega$.
\end{itemize}

\section*{Acknowledgements}

We acknowledge the use of the Lovelace (ICMAT-CSIC) cluster for research computing,
co-financed by the Spanish Ministry of Economy and Competitiveness (MINECO)
, European FEDER funds and the Severo Ochoa Programme.
J.-Ll. F. acknowledges the partial support of Essen, L. and C.-G., for
mathematical studies.
A. H. acknowledges 
support from the Spanish grant MTM2015-67724-P, and the Catalan grant 2014-SGR-1145.
A. L. acknowledges support from the Severo Ochoa Programme for Centres of Excellence
in R\&D (SEV-2015-0554), the Spanish grant MTM2016-76072-P, and
the ERC Starting Grant~335079.

\addcontentsline{toc}{section}{References}
\bibliographystyle{plain}
\bibliography{references}
\end{document}